\documentclass[12pt,a4paper]{amsart}
\usepackage{amsthm, amsfonts, amssymb, amsmath,graphicx} 
\usepackage[french, ngerman, italian, english]{babel}
\usepackage[colorlinks=true]{hyperref}
\usepackage[all]{xy}
\def\cocoa
{\mbox{\rm C\kern-.13em o\kern-.07
em C\kern-.13em o\kern-.15em A}}

\def\cocoal
{\mbox{\rm C\kern-.13em o\kern-.07
em C\kern-.13em o\kern-.15em A\kern-.1em L}}

\def\opn#1#2{\def#1{\mathop{\kern0pt\fam0#2}\nolimits}} 

\newtheorem{theorem}{Theorem}[section]
\newtheorem{lemma}[theorem]{Lemma}
\newtheorem{corollary}[theorem]{Corollary}
\newtheorem{proposition}[theorem]{Proposition}
\newtheorem{remark}[theorem]{Remark}

\theoremstyle{definition}
\newtheorem{definition}[theorem]{Definition}
\newtheorem{example}[theorem]{Example}
\newtheorem{remark/example}[theorem]{Remark/Example}

\let\oldlabel=\label
\def\prellabel{\marginparsep=1em\marginparwidth=44pt
 \def\label##1{\oldlabel{##1}\ifmmode\else\ifinner\else
 \marginpar{{\footnotesize\ \\ \tt
 ##1}}\fi\fi}}

\numberwithin{equation}{section}

\def \lk{{\operatorname{lk}}}
\def \D{\Delta}

\begin{document}

\title{  Cohen-Macaulayness of  generically complete intersection monomial ideals }
\author{\textbf{Le Dinh Nam and Matteo Varbaro}\\ \\
\textit{Dipartimento di Matematica, Universit\`a di Genova}\\
\textit{Via Dodecaneso 35, 16146 Genova, Italy}\\
\texttt{Email: ledinh@dima.unige.it, varbaro@dima.unige.it}}
\date{}
  
\maketitle
\vspace{15pt}
\begin{abstract}
\noindent In this paper we try  to understand which generically complete intersection monomial ideals with fixed radical are Cohen-Macaulay. We are able to give a complete characterization for a special class of simplicial complexes, namely the Cohen-Macaulay complexes without cycles in codimension 1. Moreover, we give sufficient conditions when the square-free monomial ideal has minimal multiplicity.
%
\end{abstract}
\section{Introduction}

Let $R=k[x_1,\dots,x_n]$ be a polynomial ring over a field $k$ and  $\Delta$ be a simplicial complex on $V=\{v_1,\dots,v_n\}$. The Stanley-Reisner ideal of $\Delta$ is:
$$ I_{\Delta}=\bigcap_{F\in \Im (\Delta)}(x_i:\; v_i \notin F) ,$$
where $\Im (\Delta)$ is the set of facets of $\Delta$. Given an ideal $J\subset R$ such that $\sqrt{J}=I_\D$, it turns out that $R/I_\D$ is Cohen-Macaulay whenever $R/J$ is Cohen-Macaulay. Of course the converse is not true, so in this paper we are going to study the following problem: \textit{How to discribe a family of ideals $J$ such that $R/J$ is Cohen-Macaulay and $\sqrt{J}=I_{\Delta}$?}

We restrict our attention on monomial ideals $J$. This problem has been already considered, for instance see the paper of Miller, Sturmfels and Yanagawa \cite{MSY}. Also, independently and with different proofs, Minh and Trung in \cite{MT} and the second author of this paper in \cite{Va}, characterized the simplicial complexes $\D$ for which all the symbolic powers of $I_\D$ are Cohen-Macaulay. However we consider a different type of family of monomial ideals with a fixed radical, namely the generically complete intersection monomial ideals:
$$I_{\Delta(\alpha)}= \bigcap_{F\in \Im (\Delta)} (x_i^{\alpha_i(F)}:\; v_i \notin F),$$
where $\alpha_i(F)$ are positive integers. In \cite{HTT}, Herzog, Takayama and Terai characterized those simplicial complexes for which $R/I_{\D(\alpha)}$ is Cohen-Macaulay for any choice of $\alpha$. It turns out that such complexes are very rare.

The purpose of this paper is to give conditions, depending on $\D$, on the values $\alpha_i(F)$ in such a way that $R/I_{\D(\alpha)}$ is Cohen-Macaulay. It is easy to see that if $\alpha_i(F)$ is constant  for any $i$, then the depth of $R/I_{\D(\alpha)}$ is equal to the depth of $R/I_\D$. However, even if $R/I_\D$ is Cohen-Macaulay, $R/I_{\D(\alpha)}$ might not be Cohen-Macaulay for ``simple" functions $\alpha$.  For instance consider the triangulation of the projective plane in the picture below (all the visible triangles are actually faces):
\begin{figure*}[!ht]
\centering
\unitlength 1pt
\begin{picture}(178.2565,183.8759)(108.1205,-237.4382)
%
\special{pn 13}%
\special{pa 2355 965}%
\special{pa 3374 974}%
\special{fp}%
\special{pa 3366 974}%
\special{pa 3963 1863}%
\special{fp}%
\special{pa 3963 1863}%
\special{pa 3359 2769}%
\special{fp}%
\special{pa 3359 2769}%
\special{pa 2340 2760}%
\special{fp}%
\special{pa 2340 2760}%
\special{pa 1750 1863}%
\special{fp}%
\special{pa 1750 1863}%
\special{pa 2348 974}%
\special{fp}%
\special{pa 2864 1440}%
\special{pa 2864 1440}%
\special{fp}%
\special{pa 2857 1432}%
\special{pa 2569 1979}%
\special{fp}%
\special{pa 2569 1979}%
\special{pa 3093 1979}%
\special{fp}%
\special{pa 3093 1979}%
\special{pa 2849 1423}%
\special{fp}%
\special{pa 2849 1432}%
\special{pa 2355 983}%
\special{fp}%
\special{pa 2857 1432}%
\special{pa 3359 983}%
\special{fp}%
\special{pa 2857 1440}%
\special{pa 3956 1872}%
\special{fp}%
\special{pa 3948 1872}%
\special{pa 3085 1961}%
\special{fp}%
\special{pa 3085 1961}%
\special{pa 3366 2787}%
\special{fp}%
\special{pa 3085 1970}%
\special{pa 2348 2760}%
\special{fp}%
\special{pa 2348 2760}%
\special{pa 2591 1961}%
\special{fp}%
\special{pa 2591 1961}%
\special{pa 1764 1854}%
\special{fp}%
\special{pa 2576 1952}%
\special{pa 2355 983}%
\special{fp}%
\put(168.5826,-61.8847){\makebox(0,0){$v_2$}}%
\put(242.7020,-67.0774){\makebox(0,0)[lb]{$v_1$}}%
\put(199.1693,-94.6054){\makebox(0,0)[lb]{$v_6$}}%
\put(108.1205,-138.7072){\makebox(0,0)[lb]{$v_3$}}%
\put(166.4486,-152.9336){\makebox(0,0)[lb]{$v_4$}}%
\put(235.4466,-149.3770){\makebox(0,0){$v_5$}}%
\put(292.2809,-135.8619){\makebox(0,0){$v_3$}}%
\put(242.7020,-207.2072){\makebox(0,0){$v_2$}}%
\put(165.3816,-215.6719){\makebox(0,0)[lb]{$v_1$}}%
\put(206.9227,-244.1958){\makebox(0,0){Simplicial complex $\Delta$}}%
\end{picture}%

\end{figure*}

\vskip 4mm

\noindent With the help of CoCoA \cite{CT} we can check that, for any vertex $i_0$ and any facet $F_0$ not containing $i_0$, we have $R/I_{\Delta(\alpha)}$ is not Cohen-Macaulay for the following $\alpha$: 
\begin{align*}
\alpha_i(F) =
\begin{cases}
2 & \text { if } i=i_0,F=F_0,\\
1& \text {otherwise}
\end{cases}
\end{align*} 

In this paper we are going to face the above problem for a special kind of simplicial complexes, namely the \textit{Cohen-Macaulay complexes without cycles in codimension 1}, which we are going to introduce in Definition \ref{df1}. In this case we give necessary and sufficient conditions on $\alpha$ for $R/I_{\D(\alpha)}$ being Cohen-Macaulay. Without entering into the details, every $\alpha_i$ has to be weakly decreasing along particular shellings (Theorem \ref{thm1}).

By similar tools, in the last section we give sufficient conditions on $\alpha$ for $R/I_{\D(\alpha)}$ to be Cohen-Macaulay when $R/I_{\D}$ has minimal multiplicity (Theorem \ref{thm2}). We will also notice that such conditions are, in general, not necessary.

Some results in this paper have been conjectured and confirmed by using the computer algebra package  CoCoA  \cite{CT}. We wish to thank Aldo Conca for suggesting the problem. We want also to thank Satoshi Murai for introducing us to Example \ref{murai}.

\section{Cohen-Macaulay complex without cycles in codimension 1}
For general facts about commutative algebra and combinatorics see the books of Bruns and Herzog \cite{BH}, Bj\"orner \cite{B}, Stanley \cite{St2} or Miller and Sturmfels \cite{MS}.

Let $V=\{v_1,\dots,v_n\}$ be a finite set. A \textit{simplicial complex} $\Delta$ on $V$ is a collection of subsets of $V$ such that $F\in \Delta$ whenever $F\subset G$ for some $G\in \Delta$, and such that $\{v_i\}\in \Delta$ for $i=1,\dots,n.$ Given finite sets $F_1,\ldots ,F_m$ the simplicial complex on $V=\cup_{i=1}^m F_i$ generated by them, i.e. consisting in all the subsets of any $F_i$, is denoted by $<F_1,\ldots ,F_m>$. The elements of a simplicial complex $\Delta$ are its \textit{faces}. Maximal faces under inclusion are called \textit{facets}. The set of facets is denoted by $\Im (\Delta)$. The \textit{dimension of a face $F$}, dim$F$, is the number $\arrowvert F\arrowvert -1$. The \textit{dimension of $\Delta$} is: \[\text{dim}\Delta=\text{max}\{\text{dim}F : F\in \Delta \}.\] 
A simplicial complex  is \textit{pure} if all its facets are of the same dimension. It is called \textit{strongly connected} if each pair $F,G\in \Im (\Delta)$ can be connected by a \textit{strongly connected sequence}, i.e. a sequence of facets $F=F_0,F_1,\dots,F_k=G$ such that $|F_i\cap F_{i+1}|=d-1$ for all $i=0,\dots,k-1$, where $\dim \D =d-1$. We will say that $\Delta$ is \textit{shellable} if it is pure and it can be given a linear order $F_1,\dots,F_m$ to the facets of $\Delta$ in a way that  $<F_i>\cap <F_1,\dots,F_{i-1}>$ is generated by a non-empty set of maximal proper faces of $<F_i>$ for all $i=2,\dots,m$. Such a linear order is called a \textit{shelling} of $\Delta$. The \textit{link} of a face $F$ of $\D$ is the simplicial complex $\lk_{\Delta}(F)=\{G:F\cup G\in \Delta,F\cap G=\emptyset\}$.

The relations between commutative algebra and combinatorics come from the \textit{ Stanley-Reisner ideal} of $\Delta$, denoted by $I_{\D}$: it is the ideal generated by all monomials $x_{i_1}\dots x_{i_s}$ such that $\{v_{i_1},\dots,v_{i_s}\}\notin \Delta$.  If the \textit{Stanley-Reisner ring} $k[\Delta]=k[x_1,\dots,x_n]/I_{\Delta}$ is a Cohen-Macaulay ring, then $\Delta$ is called a \textit{Cohen-Macaulay complex}.

The following are well known facts:
\begin{itemize}
\item[-] $\Delta$ is shellable $\Rightarrow$ $\Delta$ is Cohen-Macaulay $\Rightarrow$ $\Delta$ is pure.
\item[-] If $\Delta$ is Cohen-Macaulay, then  $\Delta$ and $\lk_{\Delta}(F)$ are strongly connected for all faces $F$ of $\Delta.$
\end{itemize}

\begin{lemma}
\label{lm1}
Let $\Delta$ be a $(d-1)$-dimensional Cohen-Macaulay complex and $F,G\in \Im (\Delta)$ with $|F\cap G|<d-1$. Then, there exists a facet $H\in \Im(\Delta)$ such that $(F\cap G)\subset (H\cap G)$ and $|H\cap G|=d-1.$
\end{lemma}
\begin{proof}
From what said above $\lk_\D(F\cap G)$ is strongly connected. Set $G'=G \setminus (F\cap G)$ and $F'=F\setminus (G\cap F)$. There exists a strongly connected sequence $F'=F_0',F_1',\dots,F_k'=G'$ of facets of $\lk_\D(F\cap G)$. Then it is enough to set $H=F_{k-1}'\cup (F\cap G)$. The lemma is proved.
\end{proof}
Let $F$ be a face of $\Delta$. Denote by $B_F$ the ideal $(x_i: v_i\notin F)$. Lemma \ref{lm1} yields the useful corollary below.\begin{corollary}
Let $\Delta$ be a $(d-1)$-dimensional Cohen-Macaulay complex with $\Im(\Delta)=\{F_1,\dots,F_m\}$. Then, for all $i=1,\dots,m$, 
$$ \bigcap_{j\not=i}B_{F_j}+B_{F_i}=\bigcap_{{j\not=i}\atop{|F_j\cap F_i|=d-1}}B_{F_j\cap F_i}. $$
\end{corollary}
\begin{proof}
For $i=1,\dots,m$ we have
$$ \bigcap_{j\not=i}B_{F_j}+B_{F_i}=\bigcap_{j\not=i}(B_{F_j}+B_{F_i})=\bigcap_{j\not=i}(B_{F_j\cap F_i}).$$
 Using Lemma \ref{lm1}, we have the corollary.  
\end{proof}

\begin{definition}
\label{df1}
Let $\Delta$ be a $(d-1)$-dimensional pure simplicial complex. We recall that the facet graph of $\D$ (see White \cite{Wh}), denoted by $G(\Delta)$, is defined as follow:
\begin{itemize}
\item [-] The set of vertices is $V(G(\Delta))=\Im (\Delta),$
\item [-] The set of egdes is $$ E(G(\Delta))=\{\{F,G\}: F,G\in \Im (\Delta) \text{ and } |F\cap G|=d-1\}.$$  
\end{itemize}

\begin{remark}
Notice that a pure simplicial complex $\D$ is strongly connected if and only if $G(\D)$ is connected.
\end{remark}

We say that  $\Delta$ is  a \textit{Cohen-Macaulay complex without cycles in codimension 1} if $\Delta$ is Cohen-Macaulay and $G(\Delta)$ is a tree.
\end{definition}
\begin{lemma}
\label{CMT}
Let $\Delta$ be a $(d-1)$-dimensional Cohen-Macaulay complex without cycles in codimension 1 and $F_1,\dots,F_k$ be a strongly connected sequence with $k\geq 2$. Then we have $ (F_k\cap F_1)\subset (F_2\cap F_1).$
\end{lemma}
\begin{proof}
We can assume  $F_1=\{v_1,\dots,v_d\},F_2=\{v_2,\dots,v_{d+1}\}$ and $k>2$. Because $G(\Delta)$ is a tree, $|F_1\cap F_k|<d-1.$
If $ (F_k\cap F_1)\not\subset (F_2\cap F_1),$ then $v_1\in F_k$. Moreover, we have $\lk_{\Delta}\{v_1\}$ is strongly connected. Set $F_1'=F_1\setminus \{v_1\}$ and $F_k'=F_k\setminus \{v_1\}$. There exists a sequence of facets of $\lk_{\Delta}\{v_1\}$, namely $F_1',F_{t_1}',\dots,F_{t_h}',F_k'$, such that $|F_1'\cap F_{t_1}'|=|F_{t_1}'\cap F_{t_2}'|=\dots=|F_{t_h}'\cap F_k'|=d-2.$ So we have the strongly connected sequence  $F_1,F_{t_1},\dots,F_{t_h},F_k,$ with $F_{t_j}=\{v_1\}\cup F_{t_j}'$ for all $j=1,\dots,h$.
On the other hand, since $G(\Delta)$ is a tree, then the sequence $F_1,F_{t_1},\dots,F_{t_h},F_k$ coincides with the sequence $F_1,F_2,\dots,F_k$. So 
$F_2=\{v_1\}\cup F_{t_1}'$. This is a contradiction.
\end{proof}
\begin{corollary}
A Cohen-Macaulay complex without cycles in codimension 1 is shellable.
\end{corollary}
\begin{proof}
Let $\Delta$ be a Cohen-Macaulay complex without cycles in codimension 1. Because $G(\Delta)$ is a tree, we can choose a linear order $F_1,\dots,F_m$ over $\Im(\Delta)$ such that $F_j$ is a free vertex of $G(\Delta)_{|\{F_1,\dots,F_j\}}$, i.e., there exists only one edge of $G(\Delta)_{|\{F_1,\dots,F_j\}}$ which contains $F_j$. By using Lemma \ref{CMT} and induction on $m$,  it is easy to show that $F_1,\dots,F_m$  is a shelling of $\Delta$. Hence, $\Delta$ is shellable.
\end{proof}
\begin{lemma}
\label{fvo}
Let $F_1,\dots,F_m$ be a shelling of a Cohen-Macaulay complex $\D$ without cycles in codimension 1. Then $F_m$ is a free vertex of $G(\Delta)$.
\end{lemma}
\begin{proof}
If $F_m$ is not a free vertex of $G(\D)$, then there exist distinct numbers $h,k<m$ such that $|F_h\cap F_m|=|F_k\cap F_m|=d-1,$ where $\dim(\D)=d-1$. But $<F_1,\dots,F_{m-1}> $ is shellable too. In particular, it is strongly connected. Then there exists a strongly connected sequence $F_h,F_{t_1},\dots,F_{t_s},F_k$, with each $t_i<m$. Therefore we have a cycle $F_h,F_{t_1},\dots,F_{t_s},F_k, F_m,F_h$ in $G(\D)$, a contradiction.
\end{proof}
\begin{definition}
\label{df2}
Let $\Delta$ be a $(d-1)$-dimensional pure simplicial complex. For any $i=1,\dots,n$ we define the graph $G^i(\Delta)$ as follow:
\begin{itemize}
\item [-] The set of vertices is $V(G^i(\Delta))=\{V_i\}\cup \{F\in \Im (\Delta): v_i\notin F\},$ where $V_i$ is a new vertex.
\item [-] The set of egdes  is
$$E(G^i(\Delta))=
 \{\{F,G\}:  |F\cap G|=d-1\}\cup $$
$$
\{\{V_i,F\}:\text { there exists a facet } G\ni v_i \text{ and } |G\cap F|=d-1\}.
$$
\end{itemize}
The graph $G^i(\Delta)$ is called the $v_i$-\textit{graph} of $\Delta.$
\end{definition}
\begin{remark}
If $\Delta$ is a Cohen-Macaulay complex,  $G(\Delta)$ and $G^i(\Delta)$ are connected for $i=1,\dots,n.$
\end{remark}
\begin{lemma}
\label{Gi}
Let $\Delta$ be a Cohen-Macaulay complex without cycles in codimension 1. Then $G^i(\Delta)$ is a tree for all $i=1,\dots,n.$
\end{lemma}
\begin{proof}
Because $G(\Delta)$ is a tree,  $G^i(\Delta)$ is not a tree if and only if there exists a strongly connected sequence of facets $F_1,\dots,F_k$ such that  $v_i\in F_1,F_k$ and $v_i\notin F_j$ for $j=2,\dots,k-1.$ But by Lemma \ref{CMT} we have $ (F_k\cap F_1)\subset (F_2\cap F_1)$. The proof is completed.
\end{proof}
\begin{example}
\label{ex1} Consider the following simplicial complex  $\Delta$:  

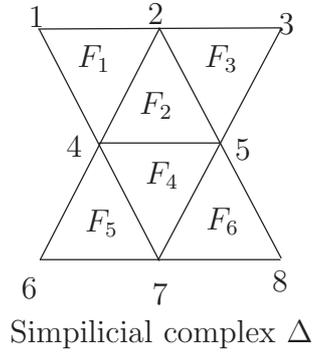
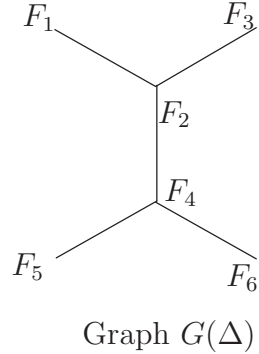
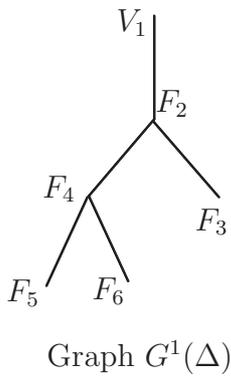
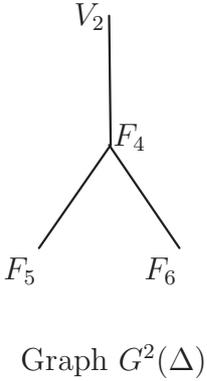
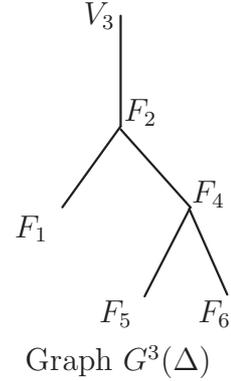
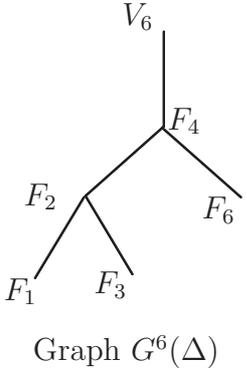
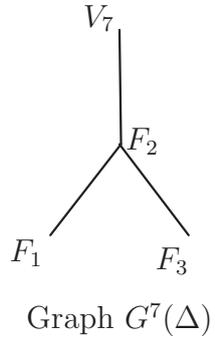
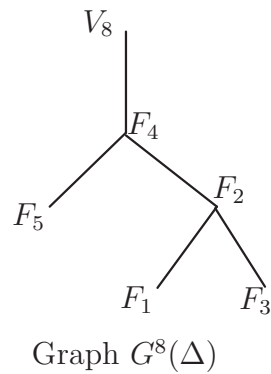
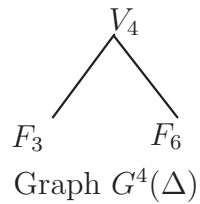
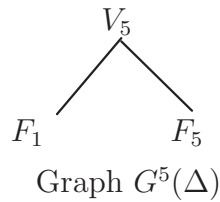
\begin{figure*}[!ht]
\centering
\unitlength 1pt
\begin{picture}(286.5193,124.6231)(  1.4226,-138.3515)
%
\special{pn 8}%
\special{pa 382 409}%
\special{pa 1002 1614}%
\special{fp}%
\special{pa 998 1609}%
\special{pa 1633 405}%
\special{fp}%
\special{pa 1633 399}%
\special{pa 377 405}%
\special{fp}%
\special{pa 1633 1600}%
\special{pa 1002 399}%
\special{fp}%
\special{pa 1002 405}%
\special{pa 386 1609}%
\special{fp}%
\special{pa 386 1609}%
\special{pa 1628 1609}%
\special{fp}%
\special{pa 1315 1000}%
\special{pa 692 1000}%
\special{fp}%
\put(48.0140,-40.5452){\makebox(0,0){$F_1$}}%
\put(71.2030,-58.1148){\makebox(0,0){$F_2$}}%
\put(95.5301,-40.3318){\makebox(0,0){$F_3$}}%
\put(73.5504,-84.5758){\makebox(0,0){$F_4$}}%
\put(50.8593,-102.5011){\makebox(0,0){$F_5$}}%
\put(96.0992,-101.8609){\makebox(0,0){$F_6$}}%
\put(73.2658,-144.3977){\makebox(0,0){Simpilicial complex $\Delta$}}%
\put(200.5208,-29.8754){\makebox(0,0)[lb]{$F_1$}}%
\put(275.2804,-29.2352){\makebox(0,0)[lb]{$F_3$}}%
\put(251.1667,-66.3661){\makebox(0,0)[lb]{$F_2$}}%
\put(253.5141,-96.0281){\makebox(0,0)[lb]{$F_4$}}%
\put(196.3240,-113.8110){\makebox(0,0)[lt]{$F_5$}}%
\put(276.7031,-118.0789){\makebox(0,0)[lt]{$F_6$}}%
\put(255.3635,-145.1091){\makebox(0,0){Graph $G(\Delta)$}}%
%
\special{pn 8}%
\special{pa 2942 408}%
\special{pa 3464 705}%
\special{fp}%
\special{pa 3464 705}%
\special{pa 3985 398}%
\special{fp}%
\special{pa 3470 699}%
\special{pa 3470 1303}%
\special{fp}%
\special{pa 3470 1303}%
\special{pa 2949 1600}%
\special{fp}%
\special{pa 3464 1309}%
\special{pa 3985 1605}%
\special{fp}%
\put(23.4735,-28.4528){\makebox(0,0)[lb]{1}}%
\put(68.1444,-27.2435){\makebox(0,0)[lb]{2}}%
\put(117.0120,-31.1558){\makebox(0,0)[lb]{3}}%
\put(37.6999,-77.5338){\makebox(0,0)[lb]{4}}%
\put(100.7939,-78.4585){\makebox(0,0)[lb]{5}}%
\put(20.7705,-130.9538){\makebox(0,0)[lb]{6}}%
\put(69.7093,-133.0878){\makebox(0,0)[lb]{7}}%
\put(114.6646,-128.3219){\makebox(0,0)[lb]{8}}%
\end{picture}%
\end{figure*}
 \begin{figure*}[!ht]
\centering
\unitlength 1pt
\begin{picture}(306.7918,132.3053)(  9.9585,-141.1968)
%
\special{pn 13}%
\special{pa 950 243}%
\special{pa 950 783}%
\special{fp}%
\special{pa 950 783}%
\special{pa 611 1184}%
\special{fp}%
\special{pa 611 1184}%
\special{pa 815 1629}%
\special{fp}%
\special{pa 604 1190}%
\special{pa 392 1653}%
\special{fp}%
\special{pa 943 792}%
\special{pa 1285 1190}%
\special{fp}%
%
\special{pn 13}%
\special{pa 950 243}%
\special{pa 950 783}%
\special{fp}%
\special{pa 950 783}%
\special{pa 611 1184}%
\special{fp}%
\special{pa 611 1184}%
\special{pa 815 1629}%
\special{fp}%
\special{pa 604 1190}%
\special{pa 392 1653}%
\special{fp}%
\special{pa 943 792}%
\special{pa 1285 1190}%
\special{fp}%
\put(54.5582,-25.1807){\makebox(0,0)[lb]{$V_1$}}%
\put(69.0691,-55.6963){\makebox(0,0)[lb]{$F_2$}}%
\put(26.8167,-87.8479){\makebox(0,0)[lb]{$F_4$}}%
\put(83.9356,-100.2960){\makebox(0,0)[lb]{$F_3$}}%
\put(13.3728,-127.2550){\makebox(0,0)[lb]{$F_5$}}%
\put(45.4533,-126.4725){\makebox(0,0)[lb]{$F_6$}}%
\put(63.3074,-146.5317){\makebox(0,0){Graph $G^1(\Delta)$}}%
\put(157.9839,-23.0467){\makebox(0,0)[lb]{$V_2$}}%
\put(172.5660,-68.5711){\makebox(0,0)[lb]{$F_4$}}%
\put(131.5940,-118.7903){\makebox(0,0)[lb]{$F_5$}}%
\put(184.2316,-118.7903){\makebox(0,0)[lb]{$F_6$}}%
\put(172.8505,-147.9543){\makebox(0,0){Graph $G^2(\Delta)$}}%
\put(263.0457,-22.4065){\makebox(0,0)[lb]{$V_3$}}%
\put(278.1257,-58.8972){\makebox(0,0)[lb]{$F_2$}}%
\put(237.2960,-102.8567){\makebox(0,0)[lb]{$F_1$}}%
\put(303.5198,-89.9818){\makebox(0,0)[lb]{$F_4$}}%
\put(268.8785,-135.1506){\makebox(0,0)[lb]{$F_5$}}%
\put(305.8671,-135.1506){\makebox(0,0)[lb]{$F_6$}}%
\put(241.1371,-154.3562){\makebox(0,0)[lb]{Graph $G^3(\Delta)$}}%
%
\special{pn 13}%
\special{pa 2368 243}%
\special{pa 2376 918}%
\special{fp}%
\special{pa 2376 918}%
\special{pa 2005 1456}%
\special{fp}%
\special{pa 2368 918}%
\special{pa 2732 1456}%
\special{fp}%
%
\special{pn 13}%
\special{pa 3831 243}%
\special{pa 3831 820}%
\special{fp}%
\special{pa 3831 820}%
\special{pa 3529 1243}%
\special{fp}%
\special{pa 3831 836}%
\special{pa 4192 1243}%
\special{fp}%
\special{pa 4192 1243}%
\special{pa 3955 1708}%
\special{fp}%
\special{pa 4186 1251}%
\special{pa 4383 1699}%
\special{fp}%
\end{picture}%
\end{figure*}
 \begin{figure*}[!ht]
\centering
\unitlength 1pt
\begin{picture}(329.1273,224.3500)( 10.6698,-235.8022)
%
\special{pn 13}%
\special{pa 1082 355}%
\special{pa 1082 850}%
\special{fp}%
\special{pa 1082 850}%
\special{pa 676 1213}%
\special{fp}%
\special{pa 676 1213}%
\special{pa 919 1622}%
\special{fp}%
\special{pa 668 1220}%
\special{pa 415 1642}%
\special{fp}%
\special{pa 1073 857}%
\special{pa 1481 1220}%
\special{fp}%
%
\special{pn 13}%
\special{pa 1082 355}%
\special{pa 1082 850}%
\special{fp}%
\special{pa 1082 850}%
\special{pa 676 1213}%
\special{fp}%
\special{pa 676 1213}%
\special{pa 919 1622}%
\special{fp}%
\special{pa 668 1220}%
\special{pa 415 1642}%
\special{fp}%
\special{pa 1073 857}%
\special{pa 1481 1220}%
\special{fp}%
\put(62.5961,-24.9673){\makebox(0,0)[lb]{$V_6$}}%
\put(79.5255,-64.3032){\makebox(0,0)[lb]{$F_4$}}%
\put(25.8920,-92.8982){\makebox(0,0)[lb]{$F_2$}}%
\put(92.4003,-98.0909){\makebox(0,0)[lb]{$F_6$}}%
\put(18.2098,-128.6776){\makebox(0,0)[lb]{$F_1$}}%
\put(51.9263,-126.1168){\makebox(0,0)[lb]{$F_3$}}%
\put(64.0187,-146.5317){\makebox(0,0){Graph $G^6(\Delta)$}}%
\put(161.4694,-25.6075){\makebox(0,0)[lb]{$V_7$}}%
\put(177.1895,-72.1277){\makebox(0,0)[lb]{$F_2$}}%
\put(133.9413,-114.3089){\makebox(0,0)[lb]{$F_1$}}%
\put(188.4995,-117.3676){\makebox(0,0)[lb]{$F_3$}}%
\put(174.9844,-134.4393){\makebox(0,0){Graph $G^7(\Delta)$}}%
\put(271.5104,-27.5280){\makebox(0,0)[lb]{$V_8$}}%
\put(288.0130,-66.0104){\makebox(0,0)[lb]{$F_4$}}%
\put(245.1916,-100.5094){\makebox(0,0)[lb]{$F_5$}}%
\put(320.0935,-90.3375){\makebox(0,0)[lb]{$F_2$}}%
\put(285.5234,-131.6651){\makebox(0,0)[lb]{$F_1$}}%
\put(330.0520,-132.3053){\makebox(0,0)[lb]{$F_3$}}%
\put(287.3728,-147.9543){\makebox(0,0){Graph $G^8(\Delta)$}}%
%
\special{pn 13}%
\special{pa 2421 346}%
\special{pa 2429 943}%
\special{fp}%
\special{pa 2429 943}%
\special{pa 2063 1418}%
\special{fp}%
\special{pa 2421 943}%
\special{pa 2779 1418}%
\special{fp}%
%
\special{pn 13}%
\special{pa 3982 355}%
\special{pa 3982 883}%
\special{fp}%
\special{pa 3982 883}%
\special{pa 3587 1269}%
\special{fp}%
\special{pa 3982 897}%
\special{pa 4453 1269}%
\special{fp}%
\special{pa 4453 1269}%
\special{pa 4145 1693}%
\special{fp}%
\special{pa 4444 1277}%
\special{pa 4702 1687}%
\special{fp}%
%
\special{pn 13}%
\special{pa 1607 2569}%
\special{pa 1289 3001}%
\special{fp}%
\special{pa 1607 2578}%
\special{pa 1934 3001}%
\special{fp}%
\put(113.9533,-185.6542){\makebox(0,0)[lb]{$V_4$}}%
\put(77.5338,-229.7560){\makebox(0,0)[lb]{$F_3$}}%
\put(129.4600,-229.0447){\makebox(0,0)[lb]{$F_6$}}%
\put(113.0997,-242.5597){\makebox(0,0){Graph $G^4(\Delta)$}}%
%
\special{pn 13}%
\special{pa 3378 2587}%
\special{pa 3045 2984}%
\special{fp}%
\special{pa 3378 2604}%
\special{pa 3744 2966}%
\special{fp}%
\put(236.5847,-185.6542){\makebox(0,0)[lb]{$V_5$}}%
\put(202.7259,-227.6220){\makebox(0,0)[lb]{$F_1$}}%
\put(262.4767,-227.6220){\makebox(0,0)[lb]{$F_5$}}%
\put(246.8277,-241.8484){\makebox(0,0){Graph $G^5(\Delta)$}}%
\end{picture}%

\end{figure*}

\end{example}

\section{The Cohen-Macaulayness for a simplicial complex without cycles in codimension 1}
Throughout this section, $\Delta$ will be a $(d-1)$-dimensional Cohen-Macaulay complex without cycles in codimension 1. Moreover the set of its facets will be $\Im (\Delta)=\{F_1,\dots,F_m\}$. The Stanley-Reisner ideal of $\Delta$ is:
$$ I_{\Delta}=\bigcap_{j=1}^m(x_i:\; v_i \notin F_j) .$$\\
For $i=1,\dots,n$, let  $\alpha_i=(\alpha_i(j): j\in \{1,\dots,m\} \text{ and } v_i \notin F_j)$ be positive integer vectors. Set 
$ Q_j=(x_i^{\alpha_i(j)}:v_i\notin F_j) $ for all $j=1,\dots,m$ and define the following ideal:
$$ I_{\Delta(\alpha)}=\bigcap_{j=1}^m Q_j. $$
Obviously, $Q_j$ is the $B_{F_j}$-primary component of $I_{\Delta(\alpha)}$ and $\sqrt{I_{\D(\alpha)}}=I_\D$.

For any vector $\textbf{a}=(a_1,\dots,a_n)\in \mathbb{N}^n$ denote by $\Delta(\alpha)_{\textbf{a}}$ the subcomplex of $\D$ with the set of facets 
$$\Im(\Delta(\alpha)_{\textbf{a}})=\{F_j\in \Im(\Delta)|a_i<\alpha_i(j) \text{ for all }i \text{ such that } v_i\notin F_j\}.$$ By \cite[Theorem 1.6]{MT}, we have:
\begin{theorem}
\label{MT}
$I_{\Delta(\alpha)}$ is Cohen-Macaulay if and only if $\Delta(\alpha)_{\textbf{a}}$ is a Cohen-Macaulay complex for all $\textbf{a}\in \mathbb{N}^n.$
\end{theorem}
Albeit Theorem \ref{MT} gives necessary and sufficient conditions for $I_{\D(\alpha)}$ to be Cohen-Macaulay, we would like to give a simpler characterization on the numbers $\alpha_i(j)$. By some experiments with CoCoA \cite{CT} on some concrete examples, we came to the followings:
\begin{definition}
Let $G$ be a tree. For any vertex $v$  of $G$, we consider the directed graph $(G,v)$ as follow:
\begin{itemize}
\item [-] The set of vertices is $V((G,v))=V(G)$.
\item [-] The pair $(u_2,u_1)\in E((G,v))$ iff there is a path $v,u_k,\dots,u_2,u_1$ in $G$.  We will call it \textit{ a directed edge } of $(G,v)$.
\end{itemize}
\end{definition}
By Lemma \ref{Gi}, $G^i(\Delta)$ is a tree for all $i=1,\dots,n.$ We have the following definition:
\begin{definition}
\label{df3}
A vector $\alpha_i=(\alpha_i(j): v_i \notin F_j)$ is called $G^i(\Delta)$-\textit{satisfying} if $\alpha_i(h)\geq \alpha_i(k)$ for all directed edges $(F_h,F_k)$ of $(G^i(\Delta),V_i).$ Moreover, $\alpha=(\alpha_i(j))$ is called \textit{$\Delta$-satisfying} if $\alpha_i$ is $G^i(\Delta)$-satisfying for all $i=1,\dots,n.$
\end{definition}
\begin{lemma}
\label{lm3}
Let $F_1,\dots,F_m$ be a shelling of  $\Delta$. If $\alpha$ is $\Delta$-satisfying, then there exists $i\in \{1,\ldots ,n\}$ and a positive integer $s$ such that
$$\bigcap_{j=1}^{m-1} Q_j+Q_m=(x_i^s)+Q_m.$$
\end{lemma}
\begin{proof}
By Lemma \ref{fvo}, $F_m$ is a free vertex of $G(\Delta)$. We can assume  $F_m=\{v_1,\dots,v_d\}$ and there exists a facet $F_h=\{v_2,\dots,v_{d+1}\}$ with $F_j\cap F_m\subsetneq F_h\cap F_m$ for all $j\not=h,m$, see Lemma \ref{CMT}.  So $F_j\cap F_m$ is a proper subset of $\{v_2,\dots,v_d\}$ for all $j\not=h,m.$ Notice that for each $ i>d+1$, the pair $(F_h,F_m)$ is a directed edge of $(G^i(\D),V_i)$. Then, because $\alpha$ is $\D$-satisfying, we have $Q_h+Q_m=(x_1^{\alpha_1(h)})+Q_m$. Moreover, $\alpha_1(h)\geq \alpha_1(j)$ for all $j\not=h,m$, since $\alpha_1$ is $G^1(\Delta)$-satisfying. Hence $(x_1^{\alpha_1(h)})\subset Q_j$ for all $j\not=h,m.$ So $Q_j+Q_m\supset Q_h+Q_m$ for all $j\not=h,m.$ We have:
$$ \bigcap_{j=1}^{m-1} Q_j+Q_m= \bigcap_{j=1}^{m-1} (Q_j+Q_m)\supseteq (Q_h\cap Q_m)=(x_1^{\alpha_1(h)})+Q_m\supseteq \bigcap_{j=1}^{m-1} Q_j+Q_m. $$
So the Lemma is proved.
\end{proof}
\begin{theorem}
\label{thm1}
 Let $\Delta$ be a Cohen-Macaulay complex without cycles in codimension 1.  Then $I_{\Delta(\alpha)}$ is Cohen-Macaulay   if and only if   $\alpha$ is $\Delta$-satisfying.
\end{theorem}
 \begin{proof} We choose a shelling $F_1,\dots,F_m$ of $\Delta$.  We denote by $\Delta_j$ the simplicial complex with the set of facets $\Im(\Delta_j)=\{F_1,\dots,F_j\}$ and $I_{\Delta_j(\alpha)}$ the ideal $\bigcap_{t=1}^jQ_t.$ We will prove the theorem by induction on $m$. This is obvious for $m=1.$ We assume that the assertion is true for $j=1,\dots,m-1.$ By Lemma \ref{fvo} we have $F_m$ is a free vertex of $G(\Delta)$. So $F_m$ is a free vertex of $G^i(\Delta)$ for all $i=1,\dots,n$ whenever $F_m$ is a vertex of $G^i(\D)$. 
If $\alpha_i$ is $G^i(\Delta)$-satisfying for all $i=1,\dots ,n$, then $(\alpha_i)_{|\Delta_{m-1}}$ is $G^i(\Delta_{m-1})$-satisfying for all $i=1,\dots ,n.$ By induction, we have $R/I_{\Delta_{k}(\alpha)}$ are $d$-dimensional Cohen-Macaulay rings  for all $k=1,\dots,m-1.$
We have the following exact sequence: 
\begin{equation}
\label{es1}
0 \rightarrow  R/I_{\Delta_{m}(\alpha)}\xrightarrow{f}  R/I_{\Delta_{m-1}(\alpha)}\oplus R/Q_{m}\xrightarrow{g} R/(I_{\Delta_{m-1}(\alpha)}+Q_{m})\rightarrow 0.
\end{equation}
By using Lemma \ref{lm3} we have  $R/(I_{\Delta_{m-1}(\alpha)}+Q_m)$ is a $(d-1)$-dimensional Cohen-Macaulay ring.
Because $R/I_{\Delta_{m-1}(\alpha)}$ and $ R/Q_m$ are Cohen-Macaulay rings of dimension $d$, we have that $R/I_{\Delta_m(\alpha)}$ is $d$-dimensional Cohen-Macaulay ring by \cite[Proposition 1.2.9]{BH}.

Conversely,  if there exists an index $i$ such that $\alpha_i$ is not $G^i(\Delta)$-satisfying, then there exists a directed edge $(F_h,F_k)$ in $G^i(\Delta)$ such that $\alpha_i(k)>\alpha_i(h).$ We choose the vector $\textbf{a}=(a_1,\dots,a_n)$ with 
$$a_t =
\begin{cases}
\alpha_i(h) & \text { if } t=i,\\
0& \text{ otherwise}.
\end{cases}$$
It turns out that if a facet $F$ of $\Delta$ contains the vertex $v_i$, then $F\in \Im(\Delta(\alpha)_{\textbf{a}})$. Moreover, $F_k\in \Im(\Delta(\alpha)_{\textbf{a}}) $ and $F_h\notin \Im(\Delta(\alpha)_{\textbf{a}})$. So, $\Delta(\alpha)_{\textbf{a}}$ is not strongly connected. Hence, $\Delta(\alpha)_{\textbf{a}}$ is not Cohen-Macaulay. This is a contradiction with Theorem \ref{MT}.
\end{proof}
\begin{example}
\label{ex2}
Let $\Delta$ be the simplicial complex of Example \ref{ex1}. 
{\small $$ I_{\Delta}=(x_3,x_5,x_6,x_7,x_8)\cap(x_1,x_3,x_6,x_7,x_8)\cap(x_1,x_4,x_6,x_7,x_8) $$
$$\cap(x_1,x_2,x_3,x_6,x_8) \cap(x_1,x_2,x_3,x_5,x_8)\cap(x_1,x_2,x_3,x_4,x_6).$$}
The ideal $I_{\Delta(\alpha)}$ is:
{\small $$ (x_3^{\alpha_3(1)},x_5^{\alpha_5(1)},x_6^{\alpha_6(1)},x_7^{\alpha_7(1)},x_8^{\alpha_8(1)})\cap(x_1^{\alpha_1(2)},x_3^{\alpha_3(2)},x_6^{\alpha_6(2)},x_7^{\alpha_7(2)},x_8^{\alpha_8(2)})$$  
$$\cap(x_1^{\alpha_1(3)},x_4^{\alpha_4(3)},x_6^{\alpha_6(3)},x_7^{\alpha_7(3)},x_8^{\alpha_8(3)}) \cap(x_1^{\alpha_1(4)},x_2^{\alpha_2(4)},x_3^{\alpha_3(4)},x_6^{\alpha_6(4)},x_8^{\alpha_8(4)})$$ 
 $$  \cap(x_1^{\alpha_1(5)},x_2^{\alpha_2(5)},x_3^{\alpha_3(5)},x_5^{\alpha_5(5)},x_8^{\alpha_8(5)})\cap(x_1^{\alpha_1(6)},x_2^{\alpha_2(6)},x_3^{\alpha_3(6)},x_4^{\alpha_4(6)},x_6^{\alpha_6(6)}) . $$}
 
\noindent Theorem \ref{thm1} tells us that $I_{\Delta(\alpha)}$ is Cohen-Macaulay if and only if  $\alpha_4(3)$, $\alpha_4(6)$, $\alpha_5(1)$ and $\alpha_5(5)$ are arbitrary positive integers and $\alpha_i(j)$ are positive integers which satisfy the order as in the following figure:
\vskip 2mm
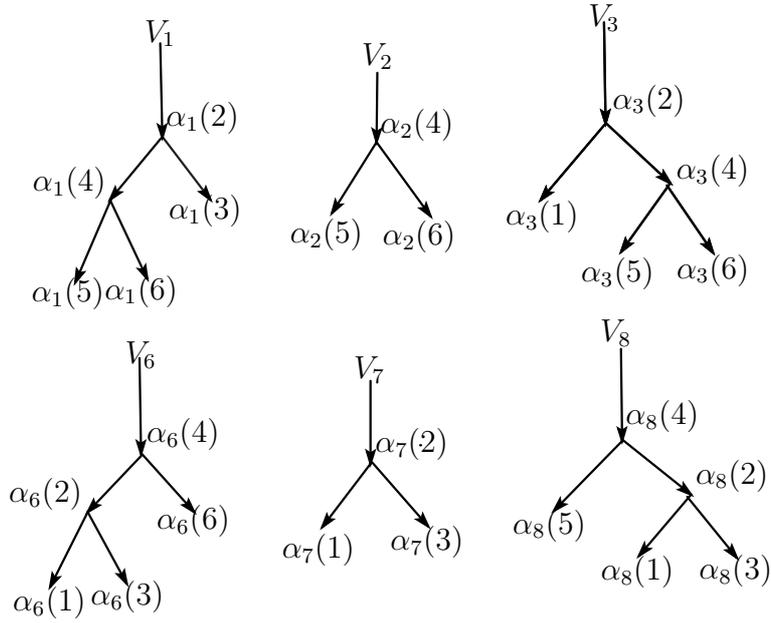
\begin{figure*}[!ht]
\unitlength 1pt
\begin{picture}(299.0385,106.6978)(  0.0000,-112.3884)
%
\special{pn 13}%
\special{pa 1282 788}%
\special{pa 1536 1124}%
\special{fp}%
\special{sh 1}%
\special{pa 1536 1124}%
\special{pa 1512 1060}%
\special{pa 1504 1081}%
\special{pa 1481 1083}%
\special{pa 1536 1124}%
\special{fp}%
\special{pa 1279 804}%
\special{pa 1016 1124}%
\special{fp}%
\special{sh 1}%
\special{pa 1016 1124}%
\special{pa 1073 1085}%
\special{pa 1050 1083}%
\special{pa 1043 1061}%
\special{pa 1016 1124}%
\special{fp}%
\special{pa 1016 1124}%
\special{pa 1209 1537}%
\special{fp}%
\special{sh 1}%
\special{pa 1209 1537}%
\special{pa 1199 1469}%
\special{pa 1187 1490}%
\special{pa 1164 1486}%
\special{pa 1209 1537}%
\special{fp}%
\special{pa 1022 1120}%
\special{pa 836 1556}%
\special{fp}%
\special{sh 1}%
\special{pa 836 1556}%
\special{pa 880 1503}%
\special{pa 857 1507}%
\special{pa 844 1488}%
\special{pa 836 1556}%
\special{fp}%
\put(92.1869,-18.4232){\makebox(0,0){$V_1$}}%
\put(94.3920,-56.4787){\makebox(0,0)[lb]{$\alpha_1(2)$}}%
\put(71.8432,-80.9481){\makebox(0,0)[rb]{$\alpha_1(4)$}}%
\put(109.0452,-85.2871){\makebox(0,0){$\alpha_1(3)$}}%
\put(85.0737,-115.6605){\makebox(0,0){$\alpha_1(6)$}}%
\put(57.6880,-116.5852){\makebox(0,0){$\alpha_1(5)$}}%
%
\special{pn 13}%
\special{pa 2396 823}%
\special{pa 2165 1194}%
\special{fp}%
\special{sh 1}%
\special{pa 2165 1194}%
\special{pa 2216 1149}%
\special{pa 2192 1150}%
\special{pa 2183 1128}%
\special{pa 2165 1194}%
\special{fp}%
\special{pa 2396 823}%
\special{pa 2676 1210}%
\special{fp}%
\special{sh 1}%
\special{pa 2676 1210}%
\special{pa 2653 1145}%
\special{pa 2645 1168}%
\special{pa 2622 1169}%
\special{pa 2676 1210}%
\special{fp}%
\put(173.4195,-27.1013){\makebox(0,0){$V_2$}}%
%
\special{pn 13}%
\special{pa 1276 307}%
\special{pa 1287 800}%
\special{fp}%
\special{sh 1}%
\special{pa 1287 800}%
\special{pa 1306 734}%
\special{pa 1286 748}%
\special{pa 1266 735}%
\special{pa 1287 800}%
\special{fp}%
\special{pa 2400 461}%
\special{pa 2396 823}%
\special{fp}%
\special{sh 1}%
\special{pa 2396 823}%
\special{pa 2417 758}%
\special{pa 2397 771}%
\special{pa 2377 757}%
\special{pa 2396 823}%
\special{fp}%
\put(174.6999,-59.0395){\makebox(0,0)[lb]{$\alpha_2(4)$}}%
\put(154.4273,-94.3209){\makebox(0,0){$\alpha_2(5)$}}%
\put(175.4824,-101.0784){\makebox(0,0)[lb]{$\alpha_2(6)$}}%
%
\special{pn 13}%
\special{pa 3574 198}%
\special{pa 3582 726}%
\special{fp}%
\special{sh 1}%
\special{pa 3582 726}%
\special{pa 3601 660}%
\special{pa 3581 674}%
\special{pa 3562 661}%
\special{pa 3582 726}%
\special{fp}%
\special{pa 3582 726}%
\special{pa 3246 1125}%
\special{fp}%
\special{sh 1}%
\special{pa 3246 1125}%
\special{pa 3303 1087}%
\special{pa 3279 1084}%
\special{pa 3273 1062}%
\special{pa 3246 1125}%
\special{fp}%
\special{pa 3582 722}%
\special{pa 3917 1037}%
\special{fp}%
\special{sh 1}%
\special{pa 3917 1037}%
\special{pa 3882 978}%
\special{pa 3878 1001}%
\special{pa 3856 1006}%
\special{pa 3917 1037}%
\special{fp}%
\special{pa 3917 1037}%
\special{pa 3663 1395}%
\special{fp}%
\special{sh 1}%
\special{pa 3663 1395}%
\special{pa 3717 1353}%
\special{pa 3693 1352}%
\special{pa 3685 1330}%
\special{pa 3663 1395}%
\special{fp}%
\special{pa 3900 1049}%
\special{pa 4138 1400}%
\special{fp}%
\special{sh 1}%
\special{pa 4138 1400}%
\special{pa 4118 1335}%
\special{pa 4109 1357}%
\special{pa 4085 1357}%
\special{pa 4138 1400}%
\special{fp}%
\put(258.2799,-12.4481){\makebox(0,0){$V_3$}}%
\put(235.0909,-87.7056){\makebox(0,0){$\alpha_3(1)$}}%
\put(285.3811,-76.3245){\makebox(0,0)[lb]{$\alpha_3(4)$}}%
\put(263.3303,-108.9741){\makebox(0,0){$\alpha_3(5)$}}%
\put(299.0385,-107.7648){\makebox(0,0){$\alpha_3(6)$}}%
%
\special{pn 13}%
\special{pa 3574 198}%
\special{pa 3582 726}%
\special{fp}%
\special{sh 1}%
\special{pa 3582 726}%
\special{pa 3601 660}%
\special{pa 3581 674}%
\special{pa 3562 661}%
\special{pa 3582 726}%
\special{fp}%
\special{pa 3582 726}%
\special{pa 3246 1125}%
\special{fp}%
\special{sh 1}%
\special{pa 3246 1125}%
\special{pa 3303 1087}%
\special{pa 3279 1084}%
\special{pa 3273 1062}%
\special{pa 3246 1125}%
\special{fp}%
\special{pa 3582 722}%
\special{pa 3917 1037}%
\special{fp}%
\special{sh 1}%
\special{pa 3917 1037}%
\special{pa 3882 978}%
\special{pa 3878 1001}%
\special{pa 3856 1006}%
\special{pa 3917 1037}%
\special{fp}%
\special{pa 3917 1037}%
\special{pa 3663 1395}%
\special{fp}%
\special{sh 1}%
\special{pa 3663 1395}%
\special{pa 3717 1353}%
\special{pa 3693 1352}%
\special{pa 3685 1330}%
\special{pa 3663 1395}%
\special{fp}%
\special{pa 3900 1049}%
\special{pa 4138 1400}%
\special{fp}%
\special{sh 1}%
\special{pa 4138 1400}%
\special{pa 4118 1335}%
\special{pa 4109 1357}%
\special{pa 4085 1357}%
\special{pa 4138 1400}%
\special{fp}%
\put(261.4808,-49.7923){\makebox(0,0)[lb]{$\alpha_3(2)$}}%
\end{picture}%
\vskip 4mm
\unitlength 1pt
\begin{picture}(316.8926,102.5011)(  2.1340,-131.6651)
%
\special{pn 13}%
\special{pa 1334 1127}%
\special{pa 1605 1431}%
\special{fp}%
\special{sh 1}%
\special{pa 1605 1431}%
\special{pa 1576 1369}%
\special{pa 1569 1391}%
\special{pa 1547 1395}%
\special{pa 1605 1431}%
\special{fp}%
\special{pa 1332 1136}%
\special{pa 1054 1431}%
\special{fp}%
\special{sh 1}%
\special{pa 1054 1431}%
\special{pa 1113 1396}%
\special{pa 1090 1392}%
\special{pa 1084 1370}%
\special{pa 1054 1431}%
\special{fp}%
\special{pa 1054 1431}%
\special{pa 1256 1805}%
\special{fp}%
\special{sh 1}%
\special{pa 1256 1805}%
\special{pa 1243 1738}%
\special{pa 1232 1758}%
\special{pa 1208 1756}%
\special{pa 1256 1805}%
\special{fp}%
\special{pa 1056 1427}%
\special{pa 859 1822}%
\special{fp}%
\special{sh 1}%
\special{pa 859 1822}%
\special{pa 906 1772}%
\special{pa 882 1775}%
\special{pa 871 1754}%
\special{pa 859 1822}%
\special{fp}%
\put(98.2331,-79.9522){\makebox(0,0)[lb]{$\alpha_6(4)$}}%
\put(95.9569,-44.1018){\makebox(0,0){$V_6$}}%
\put(73.9772,-102.6433){\makebox(0,0)[rb]{$\alpha_6(2)$}}%
\put(115.6605,-106.6267){\makebox(0,0){$\alpha_6(6)$}}%
\put(61.8136,-137.6402){\makebox(0,0){$\alpha_6(1)$}}%
\put(90.9777,-135.5063){\makebox(0,0){$\alpha_6(3)$}}%
%
\special{pn 13}%
\special{pa 1322 627}%
\special{pa 1330 1128}%
\special{fp}%
\special{sh 1}%
\special{pa 1330 1128}%
\special{pa 1349 1063}%
\special{pa 1329 1076}%
\special{pa 1310 1063}%
\special{pa 1330 1128}%
\special{fp}%
\special{pa 2775 1082}%
\special{pa 2775 1082}%
\special{fp}%
%
\special{pn 13}%
\special{pa 2518 746}%
\special{pa 2518 1179}%
\special{fp}%
\special{sh 1}%
\special{pa 2518 1179}%
\special{pa 2538 1113}%
\special{pa 2518 1126}%
\special{pa 2499 1113}%
\special{pa 2518 1179}%
\special{fp}%
\special{pa 2518 1179}%
\special{pa 2268 1510}%
\special{fp}%
\special{sh 1}%
\special{pa 2268 1510}%
\special{pa 2323 1470}%
\special{pa 2300 1468}%
\special{pa 2292 1446}%
\special{pa 2268 1510}%
\special{fp}%
\special{pa 2527 1172}%
\special{pa 2817 1506}%
\special{fp}%
\special{sh 1}%
\special{pa 2817 1506}%
\special{pa 2789 1444}%
\special{pa 2783 1467}%
\special{pa 2759 1470}%
\special{pa 2817 1506}%
\special{fp}%
\put(181.5286,-49.7923){\makebox(0,0){$V_7$}}%
\put(162.4652,-117.5810){\makebox(0,0){$\alpha_7(1)$}}%
\put(183.7337,-83.7934){\makebox(0,0)[lb]{$\alpha_7(2)$}}%
\put(203.0815,-114.7357){\makebox(0,0){$\alpha_7(3)$}}%
%
\special{pn 13}%
\special{pa 3815 579}%
\special{pa 3821 1060}%
\special{fp}%
\special{sh 1}%
\special{pa 3821 1060}%
\special{pa 3840 994}%
\special{pa 3820 1007}%
\special{pa 3801 995}%
\special{pa 3821 1060}%
\special{fp}%
\special{pa 3821 1060}%
\special{pa 3467 1425}%
\special{fp}%
\special{sh 1}%
\special{pa 3467 1425}%
\special{pa 3527 1391}%
\special{pa 3503 1387}%
\special{pa 3499 1364}%
\special{pa 3467 1425}%
\special{fp}%
\special{pa 3821 1052}%
\special{pa 4178 1340}%
\special{fp}%
\special{sh 1}%
\special{pa 4178 1340}%
\special{pa 4139 1284}%
\special{pa 4137 1307}%
\special{pa 4115 1314}%
\special{pa 4178 1340}%
\special{fp}%
\special{pa 4178 1340}%
\special{pa 3908 1660}%
\special{fp}%
\special{sh 1}%
\special{pa 3908 1660}%
\special{pa 3965 1623}%
\special{pa 3941 1620}%
\special{pa 3936 1597}%
\special{pa 3908 1660}%
\special{fp}%
\special{pa 4158 1351}%
\special{pa 4415 1670}%
\special{fp}%
\special{sh 1}%
\special{pa 4415 1670}%
\special{pa 4389 1607}%
\special{pa 4382 1628}%
\special{pa 4359 1631}%
\special{pa 4415 1670}%
\special{fp}%
\put(277.5566,-73.4081){\makebox(0,0)[lb]{$\alpha_8(4)$}}%
\put(273.7866,-35.9216){\makebox(0,0){$V_8$}}%
\put(248.8905,-108.4761){\makebox(0,0){$\alpha_8(5)$}}%
\put(303.5909,-95.6724){\makebox(0,0)[lb]{$\alpha_8(2)$}}%
\put(281.8245,-125.9034){\makebox(0,0){$\alpha_8(1)$}}%
\put(318.5997,-125.4055){\makebox(0,0){$\alpha_8(3)$}}%
\end{picture}%
\vskip 2mm
\end{figure*}
\end{example}
Of course, we can define $I_{\D(\alpha+\bf{1})}$ for any vector $\alpha \in (\mathbb{N}^n)^m$ in the obvious way. For such an $\alpha$, we say that it is $\D$-satisfying if the collection of numbers $((\alpha_i)_j+1)$ where $i=1,\ldots ,n$ and $v_i \notin F_j$ is $\D$-satisfying. 
\begin{corollary}\label{cone}
Let $\Delta$ be a Cohen-Macaulay complex without cycles in codimension 1 and $\alpha,\beta$ be vectors in $(\mathbb{N}^{n})^m$ such that $I_{\Delta(\alpha+\bf{1})},I_{\Delta(\beta+\bf{1})}$ are Cohen-Macaulay, then $I_{\Delta(\alpha+\beta+\bf{1})}$ is Cohen-Macaulay.
\end{corollary}
\begin{proof}
Because $I_{\Delta(\alpha+\bf{1})}$ and $I_{\Delta(\beta+\bf{1})}$ are Cohen-Macaulay, then  $\alpha$ and $\beta$ are $\Delta$-satisfying. Thus, $\alpha+\beta$ is $\Delta$-satisfying. So $I_{\Delta(\alpha+\beta+\bf{1})}$ is Cohen-Macaulay.
\end{proof}
Corollary \ref{cone} says that, if $\D$ is a Cohen-Macaulay complex without cycles in codimension 1, the set 
$$S=\{\alpha \in (\mathbb{N}^{n})^m \ : \ I_{\D(\alpha+\bf{1})} \mbox{ is Cohen-Macaulay}\}$$ 
is an affine semigroup. It is possible to describe a finite system of generators of $S$. 
Fixed $i\in \{1,\ldots ,n\}$, the idea is to pick the vectors $\alpha_H=((\alpha_p)_q)$, for any poset ideal $H$ of $(G^i(\D),v_i)$, such that the nonzero entries of $\alpha$ are just in $\alpha_i$ and
$$(\alpha_i)_j =
\begin{cases}
1 & \text{ if } F_j \in G^i(\D)\setminus  H, \\
0& \text{ otherwise}
\end{cases}.$$
\begin{remark} The conclusion of Corollary \ref{cone} is not true for general complexes. For instance, consider the square 
$$<\{1,2\},\{2,3\},\{3,4\},\{4,1\}>.$$
\end{remark}
\begin{corollary}
\label{prop1}
Let $\Delta$ be a Cohen-Macaulay complex without cycles in codimension 1 and 
$$\alpha_i(j) =
\begin{cases}
a_i & \text { if } i\in H,j\in K,\\
1& \text{otherwise},
\end{cases}$$
where $H$ is a subset of $[n]$, $K$ is a subset of $[m]$ and $a_i$ are integer numbers bigger than $1$ for all $i\in H$.
Then $I_{\Delta(\alpha)}$ is Cohen-Macaulay if and only if $G^i(\Delta)_{|\{V_i\}\cup\{F_j|j\in K\}}$ are trees for all $i\in H.$
\end{corollary}

\begin{proof}
If $G^i(\Delta)_{|\{V_i\}\cup\{F_j|j\in K\}}$ are trees for all $i\in H,$ we have $\alpha_i$ is $G^i(\Delta)$-satisfying for all $i=1,\dots,n.$ It implies that $I_{\Delta(\alpha)}$ is Cohen-Macaulay.\\
Conversely, if $G^i(\Delta)_{|\{V_i\}\cup\{F_j|j\in K\}}$ is not a tree for some $i$, then $\alpha_i$ is not $G^i(\D)$-satisfying. Therefore we conclude by Theorem \ref{thm1}.
\end{proof}
\section{The Cohen-Macaulayness for a strongly connected quasi-tree}
Let $\Delta$ be a $(d-1)$-dimensional simplicial complex.  Denote by $f_i$ the number of $i$-dimensional faces of $\Delta$. The vector $f(\Delta)=(f_0,f_1,\dots,f_{d-1})$ is called \textit{f-vector} of $\Delta.$
The Hilbert series of the Stanley-Reisner ring is:
$$ H_{k[\Delta]}(t)=\frac{h_0+h_1t+\cdots +h_st^s}{(1-t)^d},$$
where $s\leq d$.
The finite sequence of integers $h(\Delta)=(h_0,h_1,\dots ,h_s)$ is called the \textit{h-vector} of $\Delta.$ The multiplicity of the Stanley-Reisner ring is $e(k[\Delta])=\sum_{i=0}^sh_i.$ The $h$-vector and the $f$-vector of a simplicial complex are related by a formula. In particular, we have: 
$$ h_0=1, h_1=f_0-d \text{ and } \sum_{i=0}^sh_i=f_{d-1},$$
for instance see \cite[Corollary 5.1.9]{BH}. 
So, $e(k[\Delta])\geq1+(n-d)$ for all Cohen-Macaulay simplicial complexes $\Delta$. A Cohen-Macaulay simplicial complex has\textit{ minimal multiplicity} if $e(k[\Delta])=1+(n-d)$.

We recall the following definition. The facet $F$ of $\Delta$ is called a \textit{leaf} of $\Delta$ if there exists a facet $G$ such that $(H\cap F)\subseteq (G\cap F)$  for all $H\in \Im(\Delta).$ The facet $G$ is called a \textit{branch} of $F$. A simplicial complex $\Delta$ is called a \textit{quasi-forest} if there exists a total order $\Im (\Delta)=\{F_1,\dots,F_m\}$ such that $F_i$ is a leaf of $<F_1,\dots,F_i>$ for all $i=1,\dots,m$. This order is called a \textit{leaf order} of the quasi-forest. A connected quasi-forest is called a \emph{quasi-tree}. For properties about quasi-tree see the paper of the first author with Constantinescu \cite{CN}. Maybe the following statement is already known. However we did not find it anywhere, so we prefer to include a proof here.
\begin{proposition}\label{scqt}
Let $\Delta$ be a simplicial complex. The following conditions are equivalent:\\
(i) $\Delta$ is a strongly connected complex with minimal multiplicity;\\
(ii) $\Delta$ is a Cohen-Macaulay complex with minimal multiplicity;\\
(iii) $\Delta$ is a shellable complex with minimal multiplicity;\\
(iv) $\Delta$ is a strongly connected quasi-tree.
\end{proposition}
\begin{proof} We assume $\Delta$ is a $(d-1)$-dimensional simplicial complex with $n$ vertices and $m$ facets.

If $\Delta$ is strongly connected, we build the facets order by choosing the facet $F_i$ such that $<F_1,\dots,F_i>$ is strongly connected for all $i=1,\dots,m.$ We have $$|F_i\setminus \bigcup_{j=1}^{i-1}F_j|\leq 1,$$ for all $i=1,\dots,m.$ However, $e(k[\Delta])=1+(n-d)=m$. So, $n=d+(m-1).$ This implies $|F_i\setminus \bigcup_{j=1}^{i-1}F_j|=1$ for all $i=2,\dots,m.$ By this fact, (i),(ii),(iii) and (iv) are easily seen to be equivalent.
\end{proof}
Notice that by Proposition \ref{scqt} one can easily deduce that the notion of ``strongly connected quasi-tree" coincides with the one of ``tree" introduced in the paper of Jarrah and Laubenbacher \cite[Definition 4.4]{JL}. However, we do not call them trees because such a term is also used by other authors with a different meaning (for instance see the paper of Faridi \cite[Definition 9]{Fa}). An interesting consequence of Proposition \ref{scqt} and \cite[Theorem 4.10]{JL} is that strongly connected quasi-trees are exactly the clique complexes of a chordal graph.

\begin{remark}
(i) $\Delta$ is a Cohen-Macaulay complex without cycles in codimension 1 $\Rightarrow$ $\Delta$ is a strongly connected quasi-tree.\\
(ii) The converse is not true. For example, $\Delta=<\{1,2\},\{1,3\},\{1,4\}>$.
\end{remark}
\begin{definition}
Let $\Delta$ be a strongly connected quasi-tree with the leaf order $F_1,\dots,F_m.$ We define \textit{a relation tree of $\Delta$}, denoted by $T(\Delta)$, in the following way:
\begin{itemize}
  \item  The vertices of $T(\Delta)$ are the facets of $\Delta$.
  \item  The edges are obtained recursively as follows: 
 \begin{itemize}
  \item[-] Take the leaf $F_m$ of $\Delta$ and choose a branch $G$ of $F_m$. 
  \item[-] Set $\{F_m,G\}$ to be an edge of $T(\Delta)$. 
  \item[-] Remove $F_m$ from $\Delta$ and proceed with the remaining complex as before to determine the other edges of $T(\Delta)$.
\end{itemize}
\end{itemize}   
\end{definition}
\begin{remark} (i)The graph $T(\Delta)$ depends on the leaf order and the choice of the branch for each leaf. However it is always a tree.\\
(ii) The tree $T(\D)$ is a spanning tree of $G(\D)$.\\
(iii) If $\D$ is a Cohen-Macaulay complex without cycles in codimension 1, then the relation tree of $\D$ is $G(\D)$.
\end{remark}
\begin{lemma}
\label{mms}
Let $\Delta$ be a strongly connected quasi-tree with the relation tree $T(\Delta)$ and $F_1,F_2,\dots,F_k$ adjacent vertices in $T(\Delta).$ If the vertex $v\in F_1\cap F_k$, then $v\in F_i$ for all $i=1,\dots,k.$
\end{lemma}
\begin{proof} Let $G_1,\dots,G_m$ be the leaf order corresponding with the relation tree $T(\Delta)$ and $F_i=G_{t_i}$ for all $i=1,\dots,k$.
Because $F_1,F_2,\dots,F_k$ are adjacent vertices in $T(\Delta)$, for all $i=1,\dots,k-1$ we have:
\begin{itemize}
\item [-] If $t_i<t_{i+1},$ then $F_i$ is a branch of $F_{i+1}.$
\item [-] If $t_i>t_{i+1},$ then $F_{i+1}$ is a branch of $F_i.$
\end{itemize} 
We have two following cases:\\
\textbf{case 1:} $t_1<t_2<\cdots<t_k.$ So, $F_i$ is a branch of $F_{i+1}$ for all $i=1,\dots,k-1$. This implies $(F_1\cap F_k)\subseteq (F_{k-1}\cap F_k)$, $(F_1\cap F_{k-1})\subseteq (F_{k-2}\cap F_{k-1})$,\dots, $(F_1\cap F_3)\subseteq (F_2\cap F_3).$ Hence, $v\in F_i$ for all $i=1,\dots,k.$\\
\textbf{case 2:} $t_1>t_2>\cdots>t_h<t_{h+1}<\cdots<t_k.$ We can assume $t_1<t_k$, then $t_k$ is the biggest number in $\{t_1,\dots,t_k\}$. So, $v\in F_1\cap F_k\subseteq F_{k-1}\cap F_k.$ This implies $v\in F_1\cap F_{k-1}$. We continue with the pair $(t_1,t_{k-1})$, so on. Hence, $v\in F_i$ for all $i=1,\dots,k.$
\end{proof}
For all $i=1,\dots,n$, we define the graph $T^i(\D)$ with the set of vertices $V(T^i(\D))=V(G^i(\D))$ and the set of edge $E(T^i(\D))=E(G^i(\D))\cap E(T(\D))$. By Lemma \ref{mms}, we have:
\begin{corollary}
With the above assumptions, $T^i(\Delta)$ are trees for all $i=1,\dots,n.$ 
\end{corollary}
We consider the directed trees $(T^i(\D),V_i)$.
\begin{definition}\label{qts}
Let $\Delta$ be a strongly connected quasi-tree and $\alpha=(\alpha_i(j))$ for $i=1,\dots,n$ and $j$ such that $v_i\notin F_j.$ The collection $\alpha$ is called \textit{$\Delta$-satisfying} if there exists a relation tree $T(\Delta)$ such that if the directed edge $(F_h,F_k)\in E((T^i(\Delta),V_i))$ then $\alpha_i(h)\geq \alpha_i(k).$
\end{definition}
The proof of Lemma \ref{lm3} works also if $\D$ is a strongly connected quasi-tree. So, arguing as in the proof of Theorem \ref{thm1}, we have:
\begin{theorem}
\label{thm2}
Let $\Delta$ be a strongly connected quasi-tree and  $\alpha$ be $\Delta$-satisfying. Then, $I_{\Delta(\alpha)}$ is Cohen-Macaulay.
\end{theorem}
The converse is not true. For example, let $\Delta$ be the strongly connected quasi-tree with the set of facets:
$$\Im(\Delta)=\{\{1,6\},\{2,6\},\{3,6\},\{4,6\},\{5,6\}\}.$$ The graph $G(\Delta)$ is the complete graph on $\{F_1,\dots,F_5\}.$ The Stanley-Reisner ideal $I_{\Delta}$ is:
{\footnotesize $$  (x_2,x_3,x_4,x_5)\cap(x_1,x_3,x_4,x_5)\cap(x_1,x_2,x_4,x_5)\cap(x_1,x_2,x_3,x_5)\cap(x_1,x_2,x_3,x_4). $$}
Consider $I_{\Delta(\alpha)}$:
{\footnotesize $$ (x_2^2,x_3,x_4,x_5)\cap(x_1,x_3^2,x_4,x_5)\cap(x_1,x_2,x_4^2,x_5)\cap(x_1,x_2,x_3,x_5^2)\cap(x_1^2,x_2,x_3,x_4). $$
}
It is easy to check that $I_{\Delta(\alpha)}$ is Cohen-Macaulay but  $\alpha$ is not $\Delta$-satisfying.

\vskip 2mm

We end the paper by observing that we do not see how to extend the obtained results to more general simplicial complexes.

Given a shellable simplicial complex $\D$ with $\Im (\Delta)=\{F_1,\dots,F_m\}$, we could define a collection of positive integers $\alpha=(\alpha_i(j))$, for $i=1,\dots,n$ and $j$ such that $v_i\notin F_j$, to be \textit{$\Delta$-satisfying} if: {\it For any $i=1,\ldots ,n$ there exists a shelling $F_{i_1},\ldots ,F_{i_m}$ such that: 
\begin{enumerate}
\item There exists $p=1,\ldots ,m$ for which $v_i\in \bigcap_{h=1}^pF_{i_h}$ and $v_i \notin \bigcup_{h=p+1}^m  F_{i_h}$.
\item If $\alpha_i(i_t)>\alpha_i(i_s)$, then $t<s$.
\end{enumerate}}
\noindent It is easy to see that Definitions \ref{df3} and \ref{qts} are included in the one above. However the analog of Theorem \ref{thm2} does not hold in the general setting. For instance consider $\D$ to be the square and the collection $\alpha$ corresponding to the following ideal:
$$I_{\D(\alpha)} = (x_1,x_2^2)\cap (x_1,x_3^3)\cap (x_2^3,x_4)\cap (x_3^2,x_4).$$
Albeit $\alpha$ is $\D$-satisfying, $I_{\D(\alpha)}$ is not Cohen-Macaulay.

We can prove that $I_{\D(\alpha)}$ is Cohen-Macaulay whenever $\alpha$ is $\D$-satisfying and there is an index $i=1,\ldots ,n$ \ such that $\alpha_j$ is constant for any $j\neq i$. But this is not so nice, since in general, given a vertex of a shellable simplicial complex, we cannot find any shelling for which the first condition of the general definition of ``$\D$-satisfying" holds.

\begin{example}\label{murai}
The following example, due to Hachimori (\cite{Ha}), is a modification of the dunce hat. Consider the $2$-dimensional simplicial complex $\D$:

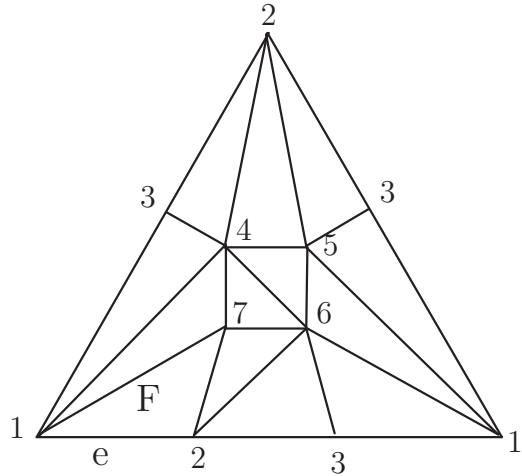
\begin{figure*}[!ht]
\centering
\unitlength 1pt
\begin{picture}(186.9346,174.6999)(113.8822,-196.5374)
%
\special{pn 13}%
\special{pa 2928 532}%
\special{pa 1726 2642}%
\special{fp}%
\special{pa 1726 2642}%
\special{pa 4135 2642}%
\special{fp}%
\special{pa 4135 2642}%
\special{pa 2921 532}%
\special{fp}%
\special{pa 2921 532}%
\special{pa 2699 1651}%
\special{fp}%
\special{pa 2699 1651}%
\special{pa 3128 1651}%
\special{fp}%
\special{pa 3128 1651}%
\special{pa 3122 2075}%
\special{fp}%
\special{pa 3128 2075}%
\special{pa 2706 2075}%
\special{fp}%
\special{pa 2706 2075}%
\special{pa 2706 1644}%
\special{fp}%
\special{pa 2706 1644}%
\special{pa 2397 1467}%
\special{fp}%
\special{pa 2914 539}%
\special{pa 3122 1644}%
\special{fp}%
\special{pa 3122 1644}%
\special{pa 3444 1452}%
\special{fp}%
\special{pa 1739 2635}%
\special{pa 2706 1636}%
\special{fp}%
\special{pa 2699 1644}%
\special{pa 3122 2068}%
\special{fp}%
\special{pa 1739 2628}%
\special{pa 2706 2062}%
\special{fp}%
\special{pa 2706 2062}%
\special{pa 2538 2642}%
\special{fp}%
\special{pa 2545 2628}%
\special{pa 3122 2075}%
\special{fp}%
\special{pa 3122 2075}%
\special{pa 3269 2622}%
\special{fp}%
\special{pa 3128 2075}%
\special{pa 4116 2628}%
\special{fp}%
\special{pa 4116 2628}%
\special{pa 3122 1651}%
\special{fp}%
\put(208.6298,-35.3525){\makebox(0,0)[lb]{2}}%
\put(169.2939,-103.9237){\makebox(0,0)[rb]{3}}%
\put(253.2295,-102.3588){\makebox(0,0)[lb]{3}}%
\put(120.2840,-191.4159){\makebox(0,0)[rb]{1}}%
\put(199.3827,-116.7274){\makebox(0,0)[lb]{4}}%
\put(231.8900,-122.3469){\makebox(0,0)[lb]{5}}%
\put(197.9600,-147.9543){\makebox(0,0)[lb]{7}}%
\put(229.4715,-147.9543){\makebox(0,0)[lb]{6}}%
\put(182.3822,-195.0436){\makebox(0,0)[lt]{2}}%
\put(234.8064,-196.5374){\makebox(0,0)[lt]{3}}%
\put(300.8168,-196.5374){\makebox(0,0)[lb]{1}}%
\put(162.0384,-181.1729){\makebox(0,0)[lb]{\large{F}}}%
\put(145.5358,-195.0436){\makebox(0,0)[lt]{\large{e}}}%
\end{picture}%

\end{figure*}

\noindent The above simplicial complex is easily seen to be shellable. However for any shelling $F_1,\ldots ,F_{13}$ we must have $F_{13}=F$. In fact $e$ is the only boundary of $\D$, so if $F_{13} \neq F$ then $\D_{12}\cap <F_{13}> = \partial F_{13}$, where $\D_{12}$ denotes the simplicial complex $<F_1,\ldots ,F_{12}>$. The Mayer-Vietories sequence yields the below long exact sequence of singular homology groups:
$$\ldots \rightarrow H_2(\D)\rightarrow H_1(\D_{12}\cap <F_{13}>)\rightarrow H_1(\D_{12})\oplus H_1(<F_{13}>)\rightarrow \ldots .$$
Because $\D_{12}$ is a $2$-dimensional shellable simplicial complex, Reisner's theorem (see \cite[Corollary 5.3.9]{BH}) implies $H_1(\D_{12})=H_1(<F_{13}>)=0$. On the other hand $H_1(\D_{12}\cap <F_{13}>)=H_1(\partial F_{13})\neq 0$. Thus the above exact sequence yields $H_2(\D)\neq 0$. But this is a contradiction, since, as it is easy to show, $\D$ is collapsible, in particular contractible.
  \end{example}

\vspace{1cm}

\end{document}